\documentclass{article}
\usepackage{amsmath}
\usepackage{amsthm}
\usepackage{amssymb}

\theoremstyle{plain}
\newtheorem{thm}{Theorem}

\newtheorem{ques}[thm]{Question}
\newtheorem{lem}[thm]{Lemma}

\theoremstyle{definition}
\newtheorem{defi}[thm]{Definition}

\theoremstyle{remark}

\newcounter{enuroman}
\renewcommand{\theenuroman}{\roman{enuroman}}
\newenvironment{romanenumerate}{\begin{list}{\rm (\theenuroman)}{\usecounter{enuroman}
    \setlength{\labelwidth}{1cm}}}
   {\end{list}}

\newcounter{enuRoman}
\renewcommand{\theenuRoman}{(\Roman{enuRoman})}

\newcounter{enuAlph}
\renewcommand{\theenuAlph}{\Alph{enuAlph}}

\newcounter{enualph}
\renewcommand{\theenualph}{\alph{enualph}}

\newcounter{enuarabic}
\renewcommand{\theenuarabic}{\arabic{enuarabic}}

\newcommand{\re}{{\upharpoonright}}

\newcommand{\A}{{\cal A}}
\newcommand{\B}{{\cal B}}

\newcommand{\F}{{\cal F}}

\newcommand{\M}{{\cal M}}

\renewcommand{\P}{{\cal P}}

\newcommand{\SSS}{{\cal S}}
\newcommand{\T}{{\cal T}}

\newcommand{\MM}{{\mathbb M}}

\newcommand{\PP}{{\mathbb P}}
\newcommand{\QQ}{{\mathbb Q}}

\newcommand{\cc}{{\mathfrak c}}

\newcommand{\rr}{{\mathfrak r}}

\newcommand{\tr}{{\mathfrak{tr}}}

\newcommand{\cov}{{\mathsf{cov}}}

\newcommand{\Spec}{{\mathsf{Spec}}}

\newcommand{\pred}{{\mathrm{pred}}}

\newcommand{\ran}{{\mathrm{ran}}}

\newcommand{\fin}{{\mathrm{fin}}}
\newcommand{\height}{{\mathrm{ht}}}

\newcommand{\cf}{{\mathrm{cf}}}

\newcommand{\Lev}{{\mathrm{Lev}}}
\newcommand{\tree}{{\mathrm{tree}}}

\newcommand{\sub}{\subseteq}
\newcommand{\sem}{\setminus}
\newcommand{\twoom}{2^\omega}
\newcommand{\twolom}{2^{<\omega}}

\newcommand{\omloms}{[\omega]^{<\omega}}

\newcommand{\omoms}{[\omega]^\omega}

\newcommand{\ha}{\,{}\hat{}\,}
\newcommand{\la}{\langle}
\newcommand{\ra}{\rangle}

\title{Maximal trees}

\author{J\"org Brendle\thanks{Partially supported by Grant-in-Aid for Scientific Research
   (C) 15K04977, Japan Society for the Promotion of Science, and by Michael Hru\v s\'ak's grants, CONACyT grant no. 177758 and 
   PAPIIT grant IN-108014. \newline
    }  \\
   Graduate School of System Informatics \\
   Kobe University \\
   Rokko-dai 1-1, Nada-ku \\
   Kobe 657-8501, Japan \\
   email: {\sf brendle@kobe-u.ac.jp} }

\begin{document}
\maketitle

\begin{abstract}
\noindent We show that, consistently, there can be maximal subtrees of $\P (\omega)$ and $\P (\omega) / \fin$
of arbitrary regular uncountable size below the size of the continuum $\cc$. We also show that there are
no maximal subtrees of $\P (\omega) / \fin$ with countable levels. Our results answer several questions of
Campero, Cancino, Hru\v s\'ak, and Miranda~\cite{CCHM16}.
\end{abstract}

%%%%%%%%%%%%%%%%%%%%%%%%%%%%%%%%%%%%%%%%%%%%%%%%%%%%%%%%%%%%%%%%%%%%%%%

%%%%%%%%%%%%%%%%%%%%%%%%%%%%%%%%%%%%%%%%%%%%%%%%%%%%%%%%%%%%%%%%%%%%%%%

\section{Introduction}

A partial order $(\T , \leq)$ is called a {\em tree} if it has a largest element ${\mathbf{1}}$ and for every $t \in \T$, the
set of {\em predecessors} of $t$ in $\T$, $\pred_\T (t) = \{ s \in \T : s \geq t \}$ is well-ordered by the reverse order of $\leq$.
For each ordinal $\alpha$, the {\em $\alpha$-th level} of $\T$ is given by $\Lev_\alpha (\T) = \{ t \in \T : \pred_\T (t)$
has order type $\alpha\}$. The {\em height} of $\T$, $\height (\T)$, is the least ordinal $\alpha$ such that $\Lev_\alpha (\T)$ is empty.
The {\em width} of $\T$ is the cardinal $\sup \{ | \Lev_\alpha (\T) | : \alpha < \height (\T) \}$. Instead of saying $\T$ has width (at most)
$\kappa$ we may sometimes just say $\T$ has levels of size $\leq\kappa$. Let $(\PP , \leq)$ be a partial order with
largest element ${\mathbf{1}}$. $\T \sub \PP$ is a {\em subtree} of $(\PP , \leq)$ (or, a {\em tree} in $\PP$) if ${\mathbf{1}}
\in \T$ and $(\T , \leq \re (\T \times \T))$ is a tree in the above sense. Note that incomparable (equivalently, incompatible) elements of
$\T$ are not necessarily incompatible in $\PP$; that is, for $s,t \in \T$ with $s \not\leq t$ and $t\not\leq s$ there may exist
$r \in \PP$ with $r \leq s,t$ (of course, such $r$ cannot belong to $\T$).

Trees are ordered by {\em end-extension}, that is, $\SSS \leq \T$ if $\SSS \sub \T$ and $\pred_\T (s) = \pred_\SSS (s)$ for 
every $s \in \SSS$. By Zorn's Lemma, {\em maximal trees}, that is, trees without proper end-extensions, exist in a given
partial order. It is easy to see that $\T \sub \PP$ is maximal iff for every $p \in \PP$
\begin{itemize}
\item either there is $q \in \T$ with $q \leq p$,
\item or there are incomparable elements $q,r \in \T$ with $p \leq q,r$.
\end{itemize}
See~\cite[Proposition 17.11]{Mo14}.

We will consider maximal trees for the case when $\PP$ is either $\P (\omega) \sem \{ \emptyset \}$, ordered by
inclusion, or $\P (\omega) \sem \{ \emptyset \} / \fin$, ordered by inclusion mod finite. For the latter recall that
for $A, B \in \omoms$, if $[A]$ and $[B]$ denote their equivalence classes in $\P (\omega) / \fin$,
then $[A] \leq [B]$ iff $A \sub^* B$ iff $ A \sem B$ is finite. We shall never work with equivalence classes
and rather consider $( \omoms , \sub^*)$ instead of  $( \P (\omega) \sem \{ \emptyset \} / \fin , \leq )$.
Monk~\cite[Proposition 17.9]{Mo14} observed that there are always maximal trees in $\P (\omega)$ of size
$\omega$ and $\cc$, and in $\P(\omega) / \fin$, of size $\cc$, and asked whether there can consistently be maximal
trees of other sizes~\cite[Problems 156 and 157]{Mo14}. These questions were solved by Campero, Cancino, Hru\v s\'ak,
and Miranda who proved that it is consistent that the continuum hypothesis CH fails and there is a maximal tree of height
and width $\omega_1$ in $\P(\omega) / \fin$~\cite[Theorem 3.2]{CCHM16} and a tree of height $\omega$
and width $\omega_1$ which is maximal as a subtree of both $\P (\omega)$ and $\P (\omega) / \fin$~\cite[Theorem 4.1]{CCHM16}.
More explicitly, the existence of such trees follows from one of the parametrized diamond principles of~\cite{MHD04},
and it is well-known that this principle is consistent with $\neg$ CH. Define the {\em tree number} $\tr$ as the least size
of a maximal tree in $\P (\omega) / \fin$ and recall that the {\em reaping number} $\rr$ (see~\cite[Definition 3.6]{Bl10}) is the
least size of a family $\A \sub \omoms$ such that for all $B \in \omoms$ there is $A \in \A$ such that either $A \cap B$
is finite or $A \sub^* B$. It is easy to see that $\omega_1 \leq \rr \leq \tr \leq \cc$~\cite[p. 81]{CCHM16}, and by the mentioned
results both $ \omega_1 = \tr < \cc$ and $\omega_1 < \tr = \cc$ are consistent. This left open the question of whether
$\tr$ can consistently be strictly in between $\omega_1$ and $\cc$~\cite[Question 5.3]{CCHM16}.

We answer this question in the affirmative by proving that for arbitrary regular uncountable $\kappa$, maximal trees in $\P (\omega) / \fin$ 
of height and width $\kappa$ can be added generically to a model with large continuum (Theorem~\ref{tr-pomegamodfin} in
Section 3 below). Furthermore, we show that, consistently, we may simultaneously adjoin maximal trees of different
sizes (Theorem~\ref{tr-spectrum}), thus making the {\em tree spectrum} $\Spec_\tree = \{ \kappa :$ there is
a maximal tree in $\P (\omega) / \fin$ of size $\kappa \}$ large and answering~\cite[Question 5.4]{CCHM16}. 
By modifying the construction, we also obtain consistently trees of width $\kappa$ and height $\omega$ which are
maximal in both $\P (\omega) / \fin$ and $\P (\omega)$, for arbitrary regular uncountable $\kappa$ (Theorem~\ref{tr-pomega}
in Section 4). Again, this construction can be extended to get large spectrum.

In all such constructions of maximal trees in $\P (\omega) / \fin$, the width is at least the cofinality of the height, and we do not know whether
there can consistently be a maximal tree of regular height whose width is smaller than its height (Question~\ref{width-height}). 
However, we prove in ZFC that there are no maximal trees in $\P (\omega) / \fin$ of countable width, thus
answering~\cite[Question 5.2]{CCHM16} (see Theorem~\ref{countable} in Section 2).

\bigskip

\noindent {\bf Acknowledgment.} This research was carried out while the author was visiting UNAM in Ciudad de M\'exico
in spring 2015. He would like to thank Michael Hru\v s\'ak for asking the questions leading to this paper, for many
stimulating discussions, and for financial support. He is also grateful to UNAM for their hospitality.

%%%%%%%%%%%%%%%%%%%%%%%%%%%%%%%%%%%%%%%%%%%%%%%%%%%%%%%%%%%%%%%%%%%%%%%

%%%%%%%%%%%%%%%%%%%%%%%%%%%%%%%%%%%%%%%%%%%%%%%%%%%%%%%%%%%%%%%%%%%%%%%

\section{Trees with countable levels}

A set $A \in \omoms$ is a {\em branching node} in a tree $\T$ if there are incomparable $B, C \in \T$ such that
$\pred_\T (B) = \pred_\T (C)$ and $A$ is the $\sub^*$-smallest node of $\pred_\T (B)$. $b \sub \T$ is a {\em maximal 
branch} if $b$ is a maximal linearly ordered subset of $\T$.

\begin{lem}
Assume $\T$ is a tree with countable levels and $b = \{ A_\alpha : \alpha < \gamma \}$ is a maximal branch in $\T$ with $\cf(\gamma) > \omega$
such that only countably many nodes $A_\alpha$ of $b$ are branching nodes. Then $\T$ cannot be maximal.
\end{lem}

\begin{proof}
Assume $\T$ is maximal.
By assumption, for some $\alpha_0 < \gamma$, no branching occurs in $b$ after $A_{\alpha_0}$. Also, by assumption, the
set $\B$ of all $B \in \T$ such that $B$ is an immediate successor of some $A_\alpha$ but $B \notin b$ must be countable.
For each $\alpha > \alpha_0$ consider the set $C_\alpha := A_{\alpha_0} \sem A_\alpha$. By maximality, there must be a
set $B_\alpha \in \B$ such that $C_\alpha \sub^* B_\alpha$ (otherwise we could add $C_\alpha$ to $\T$). By countability
of $\B$ and by $cf (\gamma) > \omega$, we see that there is a single $B \in \B$ such that for all $\alpha$,
$C_\alpha \sub^* B$. On the other hand, $A_{\alpha_0} \not\sub^* B$. In particular $A_{\alpha_0} \sem B$ is
a pseudointersection of the $A_\alpha$ (which cannot be added to the tree). Using a standard diagonal argument, we can
construct a set $C$ such that
\begin{itemize}
\item $C \sub A_{\alpha_0}$,
\item $A_{\alpha_0} \sem B \not\sub^* C$, and
\item $C \not\sub^* B'$ for all $B' \in \B$.
\end{itemize}
Now, it is easy to see that $C$ can be added to $\T$: by the third clause, the only predecessors of $C$ in $\T$ are
in $b$. By the second clause, no $A_\alpha$ is almost contained in $C$. Thus we obtain a contradiction.
\end{proof}

\begin{thm}   \label{countable}
There are no maximal trees with countable levels in $\P (\omega) / \fin$.
\end{thm}

\begin{proof}
Assume $\T$ were such a tree. Let $b= \{ A_\alpha : \alpha < \gamma \}$ be a maximal branch such that the length $\gamma$ of $b$ is minimal.
If $cf (\gamma) \neq \omega_1$, then, because of minimality and the countable levels, there can only be countably many
branching nodes in $b$. In particular, the set $\B$ of all $B \in \T$ such that $B$ is an immediate successor of some 
element of $b$ yet $B \notin b$ must be countable. If $\gamma = \delta + 1$ is a successor, a standard diagonal argument 
yields a $C \sub A_\delta$ with $A_\delta \not\sub^* C$ such that $C \not\sub^* B$ for all $B \in \B$. Similarly,
if $\gamma$ has countable cofinality, we obtain $C \sub^* A_\delta$ for $\delta < \gamma$ such that $C \not\sub^* B$ for all $B \in \B$.
In either case, $C$ can be added to $\T$, showing that $\T$ is not maximal. If $\cf (\gamma) > \omega_1$, we immediately obtain
a contradiction by the previous lemma.

So assume $cf (\gamma) = \omega_1$. By the previous lemma, using again minimality and countable levels,
we see that there is a cofinal subset of order type $\omega_1$ of branching nodes in $b$. Furthermore,
all but countably many of the branches branching off from $b$ must have length exactly $\gamma$: they cannot
be shorter by minimality, and not longer by countable levels. In particular, we may find a branching
node $A_{\la\ra} = A_{\alpha_0} \in b$ such that a branch $b'$ branching off from $b$ in $A_{\la\ra}$ has length $\gamma$.
Applying this argument again to both $b$ and $b'$, we find branching nodes $A_{\la 0 \ra} \in b$ and
$A_{ \la 1 \ra } \in b'$ above $A_{\la\ra}$. Let $\alpha_1 > \alpha_0$, $\alpha_1 < \gamma$, be such that the level of $A_{\la 0 \ra}$
and $A_{\la 1 \ra}$ is below $\alpha_1$. Iterating this procedure, we construct nodes $A_s$, $s \in\twolom$,
in $\T$ such that $A_s$ and $A_t$ are incomparable for incomparable $s$ and $t$, and $A_t \sub^* A_s$
for $t$ extending $s$. Furthermore, the level of all $A_s$, $s\in 2^n$, is below $\alpha_n < \gamma$, and
the $\alpha_n$ form a strictly increasing sequence of ordinals. Let $\alpha_\omega = \bigcup_n \alpha_n$.
Clearly $\alpha_\omega < \gamma$. Thus, by minimality, for each $f \in \twoom$ there is $A_f \in \T$ with $A_f \sub^* A_{f\re n}$
for all $n$ on level $\alpha_\omega$. In particular, the level $\alpha_\omega$ of $\T$ has size $\cc$,
a contradiction.
\end{proof}

%%%%%%%%%%%%%%%%%%%%%%%%%%%%%%%%%%%%%%%%%%%%%%%%%%%%%%%%%%%%%%%%%%%%%%%

%%%%%%%%%%%%%%%%%%%%%%%%%%%%%%%%%%%%%%%%%%%%%%%%%%%%%%%%%%%%%%%%%%%%%%%

\section{Forcing: matrix trees}

Recall that two sets $A, B \in \omoms$ are {\em almost disjoint} if $A \cap B$ is finite. $\A \sub \omoms$
is an {\em almost disjoint family} ({\em a.d. family}, for short) if any two distinct members of $\A$
are almost disjoint. If $\F \sub \omoms$ has the {\em finite intersection property}, that is, $\bigcap F$ is 
infinite for every finite $F \sub \F$, a set $C \in \omoms$ is called a {\em pseudointersection} of $\F$
if $C \sub^* A$ for all $A \in \F$.

Let $\F$ be a filter on $\omega$ containing all cofinite sets. {\em Mathias forcing} with $\F$,
written $\MM (\F)$, consists of all pairs $(s, A)$ such that $s \in \omloms$, $A \in \F$, and
$\max (s) < \min (A)$. $\MM (\F)$ is ordered by stipulating that $(t,B) \leq (s,A)$ if 
$s \sub t \sub s \cup A$ and $B \sub A$. It is well-known and easy to see that $\MM (\F)$
is a $\sigma$-centered forcing which generically adds a pseudointersection $X$ of $\F$
such that $X$ has infinite intersection with all $\F$-positive sets of the ground model.
Here $C \sub \omoms$ is {\em $\F$-positive} if $C \cap A$ is infinite for all $A \in \F$.

\begin{defi}   \label{matrixtree}
Let $\gamma$ be an ordinal. Say that a tree $\T = \{A_\alpha^\beta : \alpha, \beta \leq \gamma \}$ in $\P(\omega) / \fin$ is a 
{\em matrix tree} if 
\begin{romanenumerate}
\item for $\alpha \leq \gamma$, $\{ A_\alpha^\beta : \beta \leq \gamma \}$ is the $\alpha$-th level of $\T$,
\item for $\beta \leq \gamma$ and $\alpha < \alpha' \leq \gamma$, $A_{\alpha'}^\beta \sub^* A_\alpha^\beta$, 
\item for finite $D \sub \gamma +1$ and $\beta \notin D$, $A_0^\beta \sem \bigcup_{\beta' \in D} A_0^{\beta'}$
   is infinite, 
\item for $\alpha > 0$, $\{ A_\alpha^\beta : \beta \leq \gamma \}$ is an a.d. family, and
\item for $\beta \neq \beta '$, $A_\gamma^\beta$ and $A_0^{\beta '}$ are almost disjoint.
\end{romanenumerate}
\end{defi}

\begin{lem}  {\rm (Extension Lemma)}  \label{extension-pomegamodfin}
Assume $\T = \{A_\alpha^\beta : \alpha, \beta \leq \gamma \}$ is a matrix tree. Then there is a ccc forcing
end-extending $\T$ to a matrix tree $\T' =  \{A_\alpha^\beta : \alpha, \beta \leq \gamma + 2\}$ such that no $C \in \omoms$ from
the ground model can be added to $\T'$.
\end{lem}

\begin{proof}
Let $\F$ be a maximal filter with the property that for all $F \in \F$ and all $\beta \leq \gamma$,
$F \cap A_\gamma^\beta$ is infinite. Force with the product $\MM (\F) \times \MM (\F)$. Let $X_0$  and
$X_1$ be the two generic subsets of $\omega$. We let $A_{\gamma+1}^\beta = X_0 \cap X_1 \cap A_\gamma^\beta$.
Clearly this set is infinite by genericity. Choose $A_{\gamma + 2}^\beta \sub A_{\gamma+1}^\beta$
arbitrarily. We also let $A_0^{\gamma +1} = \omega \sem X_0$ and $A_0^{\gamma+2} = \omega \sem X_1$.
Then clearly $A_{\gamma +1}^\beta$ and $A_0^{\beta'}$ are disjoint for $\beta \leq \gamma$ and  $\beta' \in \{ \gamma+1, \gamma + 2 \}$.
A straightforward genericity argument shows that clause (iii) is still satisfied. Thus we can easily add sets $A_1^{\gamma+1}
\sub A_0^{\gamma+1}$ and $A_1^{\gamma+2} \sub A_0^{\gamma+2}$ by ccc forcing such that 
$A_0^\beta$ and $A_1^{\beta'}$ are almost disjoint for $\beta' \in \{ \gamma+1, \gamma + 2 \}$ and 
any $\beta \neq \beta'$. Finally let $\{ A_\alpha^{\beta'} : 1 < \alpha \leq \gamma + 2 \}$ be decreasing chains
below $A_1^{\beta'}$ for $\beta' \in \{ \gamma+1, \gamma + 2 \}$. It follows now that properties (iv) and (v) in the definition
of matrix tree hold. Also, $\T'$ is indeed a tree.

Let $C \in \omoms$. If $F \cap A_\gamma^\beta \sub^* C$ for some $\beta\leq\gamma$ and some $F \in \F$, 
then $A_{\gamma+1}^\beta \sub^* C$, and $C$ cannot be added to $\T'$. So assume this is not the case,
that is, $(F \cap A_\gamma^\beta) \sem C$ is infinite for all $\beta \leq\gamma$ and $F \in \F$. Then 
$\omega \sem C \in \F$ by the maximality of $\F$. Hence $X_0 \cup X_1 \sub^* \omega \sem C$
and $C \sub^* (\omega \sem X_0) \cap (\omega \sem X_1) = A_0^{\gamma +1} \cap A_0^{\gamma +2}$
and, again, $C$ cannot be added to $\T'$. This completes the proof of the lemma.
\end{proof}

Recall that $\cov (\M)$ is the least size of a family of meager sets covering the real line.
It is well-known that $\cov (\M) \leq \rr$~\cite[Theorem 5.19]{Bl10} and that adding Cohen reals increases $\cov (\M)$~\cite[Subsection 11.3]{Bl10}.

\begin{thm}   \label{tr-pomegamodfin}
Let $\kappa \leq \lambda$ be regular uncountable cardinals with $\lambda^\omega = \lambda$. There is a
ccc generic extension with $\tr = \kappa$ and $\cc = \lambda$.
\end{thm}

\begin{proof}
First add $\lambda$ Cohen reals. Then perform a finite support iteration $\la \PP_\gamma , \dot \QQ_\gamma :
\gamma < \kappa \ra$ of ccc forcing. Let $V_\gamma$ denote the intermediate model. If $\gamma$ is an even
ordinal, the model $V_{\gamma + 1}$ will contain a matrix tree $\T_\gamma = \{A_\alpha^\beta : \alpha, \beta \leq \gamma \}$
such that
\begin{itemize}
\item for $\gamma < \delta$, $\T_\delta$ end-extends $\T_\gamma$,
\item if $\gamma = \delta + 2$ is additionally a successor ordinal, then no $C$ from $V_\gamma$ can be added
   to the tree $\T_\gamma$.
\end{itemize}
If $\gamma$ is an odd ordinal, $\dot \QQ_\gamma$ is the trivial forcing. If $\gamma = \delta + 2$ is an even
successor ordinal, $\dot \QQ_\gamma$ is the forcing from the preceding lemma applied to the tree $\T_\delta
\in V_{\delta + 1} \sub V_\gamma$. If $\gamma$ is a limit ordinal, define $\dot \QQ_\gamma$ as follows:
let $\T_{<\gamma} = \bigcup_{\delta < \gamma} \T_\delta = \{A_\alpha^\beta : \alpha, \beta < \gamma \}$.
First add pseudointersections $A_\gamma^\beta$ to the decreasing chains $\{ A_\alpha^\beta : \alpha < \gamma \}$
for $\beta < \gamma$ (if $cf (\gamma) = \omega$, they can be constructed outright, otherwise they can be forced by ccc forcing).
Next add a set $A_0^\gamma$ almost disjoint from $A_0^\beta$, $\beta < \gamma$, by ccc forcing. This can be done by (iii)
and will preserve (iii) in Definition~\ref{matrixtree}. Finally let $\{ A_\alpha^\gamma : 0 < \alpha \leq \gamma  \}$ be a decreasing chain
below $A_0^\gamma$. Put $\T_\gamma = \{ A_\alpha^\beta : \alpha, \beta \leq \gamma \}$.
Then (iv) and (v) in Definition~\ref{matrixtree} clearly hold as well. This completes the definition of the iteration.

Clearly $\T_\kappa = \bigcup_{\gamma < \kappa} \T_\gamma$ is a maximal tree of size $\kappa$ by Lemma~\ref{extension-pomegamodfin}.
Therefore $\tr \leq \kappa$. On the other hand, $\tr \geq \rr \geq \cov (\M) \geq \kappa$ because of the Cohen
reals added in limit stages of the iteration.
\end{proof}

Note that the tree $\T_\kappa$ constructed in this proof has height and width $\kappa$.

\begin{thm}  \label{tr-spectrum}
Let $C$ be a set of regular uncountable cardinals. There is a ccc generic extension such that
for all $\lambda \in C$, there is a maximal tree in $\P (\omega) / \fin$ of size $\lambda$.
\end{thm}

\begin{proof}
Let $\kappa = \min C$. For $\lambda \in C$ with $\lambda > \kappa$, let $\epsilon_\lambda = \lambda \cdot \kappa$.
Make a finite support iteration $\la \PP_\gamma , \dot \QQ_\gamma : \gamma < \kappa \ra$ of ccc forcing such that
\begin{itemize}
\item if $\gamma = \delta +2$ is even successor, then $\dot \QQ_\gamma$ is defined exactly as in the proof of the previous theorem
   and end-extends the matrix tree $\T_\delta = \{ A_\alpha^\beta : \alpha,\beta \leq \delta\} \in V_{\delta + 1}$ to the 
   matrix tree $\T_\gamma = \{ A_\alpha^\beta : \alpha,\beta \leq  \gamma \} \in V_{\gamma + 1}$,
\item if $\gamma = \delta + 1$ is odd, then, for each $\lambda \in C\sem \{ \kappa\}$, $\dot\QQ_\gamma$ end-extends
   a matrix tree $\T_\delta^{\lambda} = \{ A_\alpha^\beta : \alpha,\beta \leq \lambda \cdot \delta \} \in V_\gamma$ to a
   matrix tree $\T_{\delta + 2}^{\lambda} = \{ A_\alpha^\beta : \alpha,\beta \leq \lambda \cdot (\delta + 2) \} \in V_{\gamma + 1}$ using a
   finite-support product of finite-support iterations of length $\lambda \cdot 2$ for each $\lambda \in C$ as in the proof of the previous theorem,
\item if $\gamma$ is limit, $\dot \QQ_\gamma$ end-extends $\T_{<\gamma}$ to the matrix tree $\T_\gamma$ as in the
   proof of the previous theorem and also end-extends the $\T_{<\gamma}^\lambda := \{ A_\alpha^\beta: \alpha,\beta
   < \lambda \cdot \gamma \} \in V_\gamma$ to matrix trees $\T_\gamma^\lambda = \{ A_\alpha^\beta: \alpha,\beta
   \leq \lambda \cdot \gamma \} \in V_{\gamma + 1}$.
\end{itemize}
In the final extension, let $\T_\kappa = \bigcup_{\gamma < \kappa} \T_\gamma$ and $\T_\lambda =
\bigcup_{\gamma < \kappa} \T_{\gamma}^\lambda$ for $\lambda \in C \sem \{ \kappa \}$. By construction,
all these trees $\T_\lambda$ are maximal trees, and their respective size is $\lambda$. 
\end{proof}

We do not know whether there is a way to control the $\lambda$ for which a maximal tree of size $\lambda$
is added in this proof.

\begin{ques}
Let $C$ be a set of regular cardinals (possibly satisfying some additional condition). Is there a ccc
forcing extension in which there is a maximal tree of size $\lambda$ iff $\lambda \in C$?
\end{ques}

Notice that for $\lambda > \kappa := \min C$, the trees in the previous proof all have width $\lambda = | \lambda \cdot \kappa|$
and height $\lambda \cdot \kappa$. In particular, by pruning the branches while keeping
maximality, we easily see that we can obtain maximal trees of width $\lambda$ and height $\kappa$ 
as well. Therefore, we see that all maximal trees $\T$ of regular height constructed so far either have width $|\T|$ and height $\omega$
(see Theorem~\ref{tr-pomega} below or~\cite[Theorem 4.1]{CCHM16})
or width and height $|\T|$ or width $|\T|$ and height some uncountable regular cardinal below $|\T|$. We do not
know whether there can be a maximal tree whose width is smaller than the cofinality of its height:

\begin{ques}  \label{width-height}
Is it consistent that there is a maximal tree with levels of size $\omega_1$ and height $\omega_2$
(with all branches of length $\omega_2$)?
\end{ques}

%%%%%%%%%%%%%%%%%%%%%%%%%%%%%%%%%%%%%%%%%%%%%%%%%%%%%%%%%%%%%%%%%%%%%%%

%%%%%%%%%%%%%%%%%%%%%%%%%%%%%%%%%%%%%%%%%%%%%%%%%%%%%%%%%%%%%%%%%%%%%%%

\section{Forcing: wide-branching trees}

\begin{defi}
Let $\gamma$ be an ordinal. Say that a tree $\T = \{ A_s : s \in \gamma^{< \omega} \}$ in $\P (\omega)$ is a {\em wide-branching tree}
if 
\begin{romanenumerate}
\item for all $n$, $\{ A_s : s \in \gamma^n \}$ is the $n$-th level of $\T$,
\item for $s \sub t$ in $\gamma^{< \omega}$, $A_t \sub A_s$,
\item for finite $D \sub \gamma$ and $\beta \notin D$, $A_{\la \beta \ra} \sem \bigcup_{\alpha \in D} A_{\la \alpha \ra}$ is infinite,
\item for $n \geq 2$, $\{ A_s : s \in \gamma^n \}$ is an a.d. family, and
\item for all $\alpha \leq\beta< \gamma$ and $s \in \gamma^{\geq 2}$, if $s(0) \neq \alpha$ and $\beta \in \ran (s)$ then
   $A_s$ and $A_{\la \alpha \ra}$ are almost disjoint.
\end{romanenumerate}
\end{defi}

\begin{lem} {\rm (Extension Lemma)}  \label{extension-pomega}
Let $\gamma \geq \omega$ be a limit ordinal.
Assume $\T = \{ A_s : s \in \gamma^{< \omega} \}$ is a wide-branching tree.
Then there is a ccc forcing end-extending $\T$ to a wide-branching tree $\T' = \{ A_s : s \in (\gamma + \omega)^{< \omega} \}$ 
such that for every $C \in \omoms$ from the ground model, either $A_s \sub C$ for some
$s \in (\gamma+\omega)^{<\omega}$ or $C \sub A_{\la\gamma + n\ra} \cap A_{\la\gamma +n+1\ra}$
for some $n \in\omega$.
%Additionally, for all $s \in (\gamma + 2)^{<\omega} \sem \{ \la\ra\}$ and finite $D,E s\sub \gamma$ with $s(0) \notin E$,
%$A_s \sem ( \bigcup_{\beta \in D} A_{s \ha\la\beta\ra} \cup \bigcup_{\alpha \in E} A_{\la\alpha\ra}$ is infinite.
\end{lem}

\begin{proof}
Let $\F$ be a maximal filter such that $F \cap A_s$ is infinite for all $s \in \gamma^{<\omega}$
and all $F \in \F$.
Force with the finite support product $\MM (\F)^\omega$ of countably many copies of $\MM (\F)$.
Let $(X_n : n \in\omega )$ be the generic sequence.
Put $X : = \{ m : m \in X_n$ for all $n \leq m \}$. By genericity, $X$ is an infinite pseudointersection of the $X_n$.
For each $s \in \gamma^{<\omega} \sem \{ \la\ra \}$, let
$B_s = A_s \cap X$. Note that for all finite $D, E \sub\gamma$ and all $F\in\F$ such that $s(0) \notin E$,
$F \cap A_s \sem (\bigcup_{\beta \in D} A_{s \ha\la \beta\ra} \cup \bigcup_{\alpha\in E} A_{\la\alpha\ra})$ is infinite.
(To see this, take $\delta > \max E$, $\delta \notin D$. Then $A_{s\ha\la\delta\ra}$ is almost
disjoint from $\bigcup_{\beta \in D} A_{s \ha\la \beta\ra} \cup \bigcup_{\alpha\in E} A_{\la\alpha\ra}$ 
by (iv) and (v), and $F \cap A_{s\ha\la\delta\ra}$ is infinite.) Thus, by genericity, for all finite $D,E \sub \gamma$ with $s(0) \notin E$, 
$B_s \sem (\bigcup_{\beta \in D} A_{s \ha\la \beta\ra}
\cup \bigcup_{\alpha\in E} A_{\la\alpha\ra})$ is infinite. In particular, by a further ccc forcing,
we can add pairwise disjoint sets $A_{s\ha\la\gamma+n\ra}$, $n \in\omega$, contained in $B_s$
such that all of them are almost disjoint from all $A_{s \ha\la \beta\ra}$, $\beta <\gamma$,
and all $A_{\la\alpha\ra}$, $\alpha < \gamma$, with $s(0) \neq\alpha$. This means that clauses (iv) and (v) still hold
for these sets. In particular, they can be added to the tree (that is, they are neither above an element of
the tree, nor below two incomparable elements of the tree). Furthermore, for any $s \in \gamma^{<\omega}\sem \{ \la\ra \}$,
any $n \in\omega$, and any $t \in (\gamma + \omega)^{<\omega}$, we can now
build sets $A_{s \ha \la\gamma + n\ra\ha t}$ contained in $A_{s\ha\la\gamma + n\ra}$ such that all the clauses
are still satisfied. 
%We may additionally assume that for all $n \in\omega$, there is $\beta < \gamma + 2$
%such that $A_{s \ha\la\delta\ra \ha\la\beta\ra} \sub A_{s\ha\la\delta\ra} \sem n$.

Next, let $A_{\la\gamma + n\ra} = \omega \sem X_n$ for $n\in\omega$.
These sets are almost disjoint from any $A_s$ with $s(0) < \gamma$ and $\gamma + n$ 
belonging to $\ran (s)$ for some $n \in \omega$. In particular, property (v) is preserved.
Also, property (iii) still holds by genericity. Hence an additional ccc forcing
adds sets $B_{\la\gamma + n\ra} \sub A_{\la\gamma + n\ra}$
such that $B_{\la\gamma +n\ra}$ and $A_{\la\beta\ra}$ are almost disjoint for $n \in\omega$
and any $\beta \neq\gamma + n$ with $\beta <\gamma+\omega$. Let $\{ A_{\la\gamma + n \ra\ha\la\alpha\ra} : \alpha < \gamma + \omega \}$
be an a.d. family below $B_{\la\gamma + n\ra}$ for $n \in\omega$. More generally, 
for any $n \in\omega$ and any $t \in (\gamma + \omega)^{<\omega}$, we can build sets $A_{\la\gamma + n\ra\ha t}$ contained in
$A_{\la\gamma + n\ra}$ such that all the clauses are still satisfied. This completes the definition of $\T'$,
and it is clear $\T'$ is a wide-branching tree. 

Let $C \in \omoms$. If $F \cap A_s \sub^* C$ for some $F \in \F$ and $s\in \gamma^{<\omega}$,
then $B_s \sub^* C$. In particular, for some $n \in\omega$, $A_{s \ha\la \gamma + n\ra } \sub C$.
Hence, assume that $(F\cap A_s) \sem C$ is infinite for all $F\in\F$ and $s \in \gamma^{<\omega}$. Then 
$\omega \sem C \in \F$ by maximality of $\F$. Therefore $X_n \sub^* \omega\sem C$ for all $n \in \omega$
and, by genericity, there is in fact an $n \in \omega$ such that $X_n \cup X_{n+1} \sub \omega \sem  C$.
Therefore, $C \sub A_{\la\gamma + n\ra} \cap A_{\la\gamma +n +1\ra}$ as required. 
\end{proof}

\begin{thm}  \label{tr-pomega}
Let $\kappa \leq \lambda$ be regular uncountable cardinals with $\lambda^\omega = \lambda$. There is a
ccc generic extension with $\tr = \kappa$, $\cc = \lambda$, and, additionally, there is a maximal tree in
$\P (\omega)$ of size $\kappa$.
\end{thm}

\begin{proof}
Add $\lambda$ Cohen reals and then make a finite support iteration $\la \PP_\gamma , \dot \QQ_\gamma : \gamma < \kappa \ra$
of ccc forcing as in the proof of Theorem~\ref{tr-pomegamodfin}. Let $V_\gamma$ be the intermediate model.
If $\gamma$ is a limit ordinal, the model $V_\gamma$ will contain a wide-branching tree $\T_\gamma =
\{ A_s : s \in \gamma^{<\omega} \}$ such that
\begin{itemize}
\item for $\gamma < \delta$, $\T_\delta$ end-extends $\T_\gamma$,
\item for $C \in \omoms \cap V_\gamma$ either $A_s \sub C$ for some $s \in (\gamma + \omega)^{<\omega}$
   or $C \sub A_{\la \gamma + n \ra} \cap A_{\la \gamma + n + 1 \ra}$ for some $n \in \omega$ (in the model
   $V_{\gamma + 1}$ which contains the tree $\T_{\gamma + \omega}$).
\end{itemize}
If $\gamma$ is a successor ordinal, $\dot \QQ_\gamma$ is the trivial forcing. If $\gamma$ is a limit ordinal,
$\dot \QQ_\gamma$ is the forcing from the preceding lemma applied to the tree $\T_\gamma \in V_\gamma$
and yielding the tree $\T_{\gamma + \omega} \in V_{\gamma + 1 }$. Here $\T_\gamma$ is obtained as follows:
if $\gamma = \delta + \omega$, then $\T_\gamma \in V_{\delta + 1}$ has been constructed earlier; if
$\gamma$ is a limit of limits, then $\T_\gamma = \bigcup_{\delta < \gamma} \T_\delta$. 

Clearly, $\T_\kappa = \bigcup_{\gamma < \kappa} \T_\gamma$ is a maximal tree in $\P (\omega)$ of
size $\kappa$ by Lemma~\ref{extension-pomega} which is additionally maximal in $\P (\omega) / \fin$. 
$\tr = \kappa$ follows as in the proof of Theorem~\ref{tr-pomegamodfin}.
\end{proof}

As with Theorem~\ref{tr-spectrum}, the previous result can be extended to yield big spectrum
for the size of maximal trees in $\P (\omega)$.

\begin{thm}  \label{tr-pomega-spectrum}
Let $C$ be a set of regular uncountable cardinals. There is a ccc generic extension such that
for all $\lambda \in C$, there is a maximal tree in $\P (\omega)$ of size $\lambda$ which is additionally
maximal in $\P (\omega) / \fin$.
\end{thm}

\begin{proof}
Combine the proofs of Theorems~\ref{tr-spectrum} and~\ref{tr-pomega}.
\end{proof}

%%%%%%%%%%%%%%%%%%%%%%%%%%%%%%%%%%%%%%%%%%%%%%%%%%%%%%%%%%%%%%%%%%%%%%%

%%%%%%%%%%%%%%%%%%%%%%%%%%%%%%%%%%%%%%%%%%%%%%%%%%%%%%%%%%%%%%%%%%%%%%%

\end{document}